\documentclass[draft, 12pt]{elsarticle}

\usepackage{amsmath, amsthm}             
\usepackage{amssymb}             
\usepackage{dsfont}              
\usepackage{stmaryrd}            
\usepackage{geometry}
\usepackage{bm}                  
\usepackage{enumerate}           
\usepackage{hyperref}            
\usepackage[perpage]{footmisc}   
\usepackage{graphicx}            
\usepackage[mathscr]{eucal}
\newtheorem{Theorem}{Theorem}[section]
\newtheorem{Proposition}{Proposition}[section]
\newtheorem{Lemma}{Lemma}[section]
\newtheorem{Corollary}{Corollary}[section]
\newtheorem{Definition}{Definition}[section]
\newtheorem{Remark}{Remark}[section]

\newtheorem{Notation}{Notation}

\newcommand{\w}{{\mathbf w}}

\newcommand{\R}{{\mathbb R}}

\newcommand{\N}{{\mathcal N}}

\newcommand{\x}{{\mathbf x}}

\journal{}

\begin{document}

\begin{frontmatter}



\title{Maximal Invariants For Lorentz Wishart Models}


\author{Emanuel Ben-David}

\address{Department of Statistics, Stanford University}

\begin{abstract}
In this paper we consider two statistical hypotheses for the families of Wishart type distributions. These distributions are analogs of the Wishart distributions defined and parametrized over a Lorentz cone. We test these hypotheses by means of maximal invariant statistics which are explicitly derived in the paper. The testing problems, respectively, concern the hypothesis that parameters are in a sub-Lorentz-cone, and the the hypothesis that two observations have the same parameter. 

\end{abstract}


\end{frontmatter}



\section{Introduction}
The primary aim of this paper is a detailed study of maximal invariant statistics for a family of Lorentz type Wishart distributions, or simply a Lorentz Wishart model.  A Lorentz type Wishart distribution, similar to the classical Wishart distribution \cite{Muirhead1982}, naturally arises as the distribution of the empirical sample covariance of a multivariate normal distribution \cite{Jensen1988}, or directly as the joint spectral density of observable variables \cite{Feinsilver2001}. Generally, by employing the theory of Euclidean simple algebras, the Wishart distributions can  be defined on each irreducible symmetric cone\cite{Casalis1996}, or more generally, on each homogeneous cone, by employing the theory of Vinberg algebras \cite{Andersson2004}.  According to the classification of irreducible (symmetric) cones \cite{Faraut1994} there are five types of irreducible cones. The first there types of irreducible cones are, respectively,  ${\mathrm PD}_{n}(\R), {\mathrm PD}_{n}(\mathbb{C})$ and ${\mathrm PD}_{n}(\mathbb{H})$, i.e., the cone of positive definite matrices over the field of real, complex and quaternion numbers. The forth type of irreducible cone is the Lorentz  cone (or Mikowski cone), and the fifth type is the exceptional cone over the Octonion $\mathbb{O}$.  Therefore, correspondingly, the three first types of Wishart distributions are real, complex and quaternion Wishart distributions which have been well studied in statistics and probability literature. In \cite{Jensen1988} Jensen considered statistical hypotheses for Lorentz Wishart models and obtain a complete solution to the problems of maximum likelihood inference. \\ 

\noindent
In this paper,\footnote{ This paper is a revised version of the second part of the author's doctoral thesis \cite{Bendavid2008}. In the first part, using a uniform approach, the maximal invariant statistics over irreducible cones are studied.} we explicitly derive maximal invariant statistics for testing two invariant statistical hypotheses for the family of Lorentz type Wishart distributions. The first hypothesis is studied in \cite{Jensen1988} too, but our approach is more general and it highlights the role of the maximal invariant statistic. The second hypothesis is an analog of Bartlett's test. \\
\noindent
The organization of this paper is as follows. In \S \ref{sec:pre} we give a precise definition of a Lorentz cone  $\mathcal{L}$, the description of the automorphism group $G$ of the Lorentz cone ${\mathcal{L}}$, and the definition of the Lorentz type Wishart distribution $\mathcal{W}_{\eta,\sigma}^{\mathcal{L}}$. In \S\ref{sec:test1} we test the hypothesis that the scale parameter $\sigma$ is in a Lorentz subcone
$\mathcal{L}_{0}\subset\mathcal{L}$.  To this end, first we identify a subgroup $G_{0}\subset G$ that acts transitively on $\mathcal{L}_{0}$ and derive a maximal invariant statistic associated with the hypothesis under the action of $G_{0}$.  In \S\ref{sec:bart} we test the hypothesis that two observed Lorentz type Wishart distributions have the same scale parameter $\sigma$. For this testing problem we derive a maximal invariant statistic associated with the hypothesis under the action of $G$.   
\section{Preliminaries}\label{sec:pre}
 

\subsection{The Lorentz cone}\label{sub:L_cone}
Let $W\not=\{0\}$ be a Euclidean vector space  and  let
\[
\Psi_W(\w,\w)=(\w,\w^\prime),\   \w,\w^\prime\in W,
\]
 denote the
inner product on  $W$. Set as usual $\|\w\|^2:=(\w,\w)$,  for each $\w\in W$. Consider the symmetric form  $\Psi$ on $\R\times W$ given by $\Psi((\lambda, \mathbf{w}),(\lambda, \mathbf{w})):=\lambda^2-\|w\|^{2}$.
This is a non-singular symmetric from with signature $ (1 ,{\mathrm Dim}(W))$. The symmetric form  $\Psi$ partitions according to the decomposition $\R\times W$ as

\[
\Psi=\left(\begin{array}{cc}
\psi_0&0\\
0&-\Psi_W
\end{array}\right)
\]
where $\psi_0\in\R_{+}$\footnote{we use the notations $\R_{++}$  and  $\R_{++}$, respectively,  for the set of non-negative and positive real numbers}.
The set
\[
\mathcal{L}:=\{(\lambda, \mathbf{w})\in\R\times W| \ \lambda>0,\Psi((\lambda, \mathbf{w}))>0\}
\]
is a symmetric  cone, called the Lorentz cone (generated by $W$).
\subsection{The automorphism group}\label{sub:aut}
\noindent
Define for any symmetric form  $\Psi$ on a Euclidean space $V$ the orthogonal group
\[
{\mathrm O}(\Psi):=
\left\{A\in{\mathrm GL}(V)| \ \Psi\circ(A\times A)=\Psi\right\},
\]
the special orthogonal group
\[
{\mathrm SO}(\Psi):= \left\{A\in{\mathrm O}(V)| \ {\mathrm det}(A)=1\right\}.
\] 
The connected component of the identity in ${\mathrm O}(\Psi)$, denoted by ${\mathrm SO}^{\shortuparrow}(\Psi)$, is a subgroup of ${\mathrm SO}(\Psi)$. If $\Psi$ is positive definite, and hence an inner product on $V$, then ${\mathrm SO}^{\shortuparrow}(\Psi)={\mathrm SO}(\Psi)$.

\bigskip
\noindent
Let $G_W= G$  denote the connected component of the automorphism group\  ${\mathrm Aut}(\mathcal{L})$. Then for  the symmetric form $\Psi$  defined in \S \ref{sub:L_cone} we have 
\[
G=\R_{++}\times {\mathrm SO}^{\shortuparrow}(\Psi),
\]
where 
\begin{equation}\label{S0+}
{\mathrm SO}^{\shortuparrow}(\Psi)= \left\{A=\left(\begin{array}{cc}
a_0&\mathbf{a}_{0W}\\
\mathbf{a}_{W0}&A_W
\end{array}\right)\in{\mathrm SO}(\Psi)| \ a_0>0\right\}
\end{equation}
with $\mathbf{a}_{0W}:W\rightarrow\R$, $\mathbf{a}_{W0}:\R\rightarrow W$ and $A_W:W\rightarrow W$
being linear mappings. Note that  $\R_{++}\times{\mathrm SO}^{\shortuparrow}$ acts transitively on $\mathcal{L}$ and for $(a,A)\in\R_{++}\times{\mathrm SO}^{\shortuparrow}(\R\times W)$ we have
\[
(a,A)(\lambda, \mathbf{w})=aA\left(\begin{array}{c}
\lambda\\
\w
\end{array}\right).
\]


 \subsection{The Lorentz type Wishart distributions }\label{sub:Wish_L}
 \noindent
Fix an element $e\in W$. Let $W_{1}$  be the orthogonal complement of  $\R e$, i.e., ${\R e}^{\perp}$. Thus $W=\R e\oplus W_{1}$ can be identified with  $\R\times W_{1}$. Under this identification, the Lorentz  cone $\mathcal{L}$  is isomorphic to the homogenous cone
 \[
 \mathcal{P}_2(W):=\{S=
                       \begin{pmatrix}
                         \lambda_1 & \w_{1} \\
                         \w_{1} & \lambda_2\\
                       \end{pmatrix}: \ \lambda_1, \lambda_2\in \R_{++},  \ \w_{1}\in W_{1}, \:\det(S):=\lambda_1\lambda_2-\|\w_{1}\|^2>0 \},
\]
studied in \cite{Andersson2004}.
The isomorphism  is given by
\begin{eqnarray}\label{eq:Andersson cone}
  \phi:\mathcal{P}_2(W) &\rightarrow& \mathcal{L}\nonumber\\
  \left(
    \begin{array}{cc}
      \lambda_1 & \w_{1} \\
      \w_{1} & \lambda_2 \\
    \end{array}
  \right)
  &\mapsto& (\dfrac{\lambda_1+\lambda_2}{2},(\dfrac{\lambda_1-\lambda_2}{2}, \w_{1})),
\end{eqnarray}
where the inverse mapping is  given by
\begin{eqnarray}\label{eq:Jensen cone}
 \phi^{-1}:\mathcal{L}&\rightarrow& \mathcal{P}_2(W)\nonumber\\
(\lambda_{1},(\lambda_{2},s))&\mapsto &\left(
                                \begin{array}{cc}
                                  \lambda_1+\lambda_2 & \w_{1} \\
                                  \w_{1} & \lambda_1-\lambda_2 \\
                                \end{array}
                              \right).
\end{eqnarray}
By using their general theory of the Wishart distributions for homogeneous cones,  in \cite{Andersson2004}  Andersson et al. derive the Wishart distribution  on $\mathcal{P}_2(V)$  as
  \begin{equation*}
  d\mathcal{W}_{\mathcal{P}_2(V)}( S| \eta,\Sigma)=\dfrac{\eta^{2\eta}{\mathrm det}(S)^{\eta-\frac{m+1}{2}}}
  {\pi^{\frac{m-1}{2}}\Gamma(\eta)\Gamma(\eta-\frac{m-1}{2}){\mathrm det}(\Sigma)^\eta}
  \exp\{-\eta
  {\mathrm tr}(\Sigma^{-1}S)\}\mathbf{1}_{\mathcal{P}_2(V)}(S)dS,
 \end{equation*}
 where  $\eta>\dfrac{m}{2}$  is  shape parameter, $m={\mathrm dim}_{\R}(W)$, and $\Sigma \in \mathcal{P}_2(W)$  is  the expectation  vector. Therefore, one natural way to define the Wishart distribution on $\mathcal{L}$ is via the image of $\mathcal{W}_{\mathcal{P}_2(V)}( S| \eta,\Sigma)$ under the mapping $\phi$. 
 \begin{Definition}
 The Lorentz type Wishart distribution, denoted by $\mathcal{W}^{\mathcal{L}}_{\eta, \sigma}$, is the image of the Wishart distribution $\mathcal{W}_{\mathcal{P}_2(V)}( S| \eta,\phi^{-1}(\sigma))$ under the mapping $\phi$ in Eq. \eqref{eq:Andersson cone}. One can check that the Wishart distribution on  $\mathcal{L}$ with shape parameter  $\eta>\dfrac{m-1}{2}$  and expectation   $\sigma=(\lambda, \mathbf{w})\in \mathcal{L}$  is given by
  \begin{equation}\label{Wishart}
        d\mathcal{W}^{\mathcal{L}}_{\eta,\sigma}(y,\mathbf{z})=\dfrac{(y^2-\|\mathbf{z}\|^2)^{\eta-\frac{m+1}{2}}}
  {k(m,\eta)(\lambda^2-\|\w\|^2)^\eta}
  \exp\{-2\eta(\frac{\lambda y-\w\cdot\mathbf{z}}{\lambda^2-\|\w\|^2})\}1_{\mathcal{L}}(y,\mathbf{z})dyd\mathbf{z},
\end{equation}
where  $(y,\mathbf{z})\in \R\times W$   and  $k(m,
\eta)=2\pi^{\frac{m-1}{2}}\Gamma(\eta)\Gamma(\eta-\frac{m-1}{2})\eta^{-2\eta}$.
\end{Definition}
\  $\eta=\frac{N}{4}$.
\begin{Remark}
First note that the density  given in Eq. \eqref{Wishart} is the same as formula $(29)$ in \cite{Jensen1988}. Moreover, for every irreducible symmetric cone $\Omega$, the Wishart distribution on $\Omega$, denoted by $\mathcal{W}_{\eta, \sigma}^{\Omega}$,  is well defined and given by the density
\begin{equation}\label{eq:wish_omega}
d\mathcal{W}_{\eta,\sigma}^{\Omega}(\x)=\dfrac{1}{2^{\frac{1}{2}\eta r}\Gamma_{\Omega}(\frac{\eta}{2})\det(\sigma)^{\frac{1}{2}\eta}}\exp\{-\frac{1}{2}tr(\sigma^{-1}\x)\}
\det(\x)^{\frac{1}{2}\eta-\frac{n}{r}}1_{\Omega}(\x),
 \end{equation}
where $r$ is the rank of $\Omega$, and $\Gamma_{\Omega}(\cdot)$ is the gamma function associated with $\Omega$ (see \cite{Letac2004} or \cite{Bendavid2010} for detail).  Since $\mathcal{L}$  is an irreducible symmetric cone of rank $2$, Eq. \eqref{eq:wish_omega} directly defines the Wishart distribution on $\mathcal{L}$.
\end{Remark}

 \section{Testing for scale parameter in a Lorentz subcone}\label{sec:test1}
 \noindent
  Suppose that shape parameter $\eta$ is known and consider the Lorentz Wishart model
  \begin{equation}\label{eq: Model}
  \mathcal{M}\equiv \left\{\mathcal{W}^{\mathcal{L}}_{\eta,\sigma}:\: \sigma \in \mathcal{L}\right\}.
  \end{equation}
   The standard theory of exponential families implies that the
\textbf{ML} estimator $\widehat{\sigma}_{mle}$ of $\sigma \in \mathcal{L}\in\mathcal{M}$  exists for any observation  $(y,\mathbf{z})\in \mathcal{L}$  and is given by
\begin{equation}
\widehat{\sigma}_{mle}((y,\mathbf{z}))=(y,\mathbf{z}).
\end{equation}
Now suppose  that $W_0\neq\{0\}$ is a subspace of $W$. Let $\mathcal{L}_{0}$ denote the Lorentz cone generated by $W_{0}$. Under the inclusion mapping   $\iota:\mathcal{L}_{0}\rightarrow \mathcal{L}$
the Lorentz Wishart model
\begin{equation}\label{eq:submodel}
\mathcal{M}_0 \equiv \left\{\mathcal{W}^{\mathcal{L}}_{\eta,\iota(\sigma_1)}:\:\sigma_1 \in \mathcal{L}_{0}\right\}
\end{equation}
 is a submodel of  $\mathcal{M}$. Let us consider the hypothesis 
\begin{equation}\label{eq:test1}
\tag{T1}H_0: \sigma\in \mathcal{L}_{0}\quad\text{vs.}\quad H: \sigma\in\mathcal{L}.
\end{equation}
 We will test the hypothesis \eqref{eq:test1} by a maximal invariant statistic we shall derive  in \S\ref{sub:test1}.  First we identify a subgroup $G_{0}\subset G$ such that the hypothesis is invariant under it. We proceed as follows.
\subsection{The subgroup of $G$ with invariant action on $\mathcal{L}_{0}$}\label{sub:subcone}
\noindent
Recall the definition of the symmetric form  $\Psi$  and the inner product $\Psi_{W}$  in \S \ref{sub:L_cone}. Let $W_0^\perp$ be the orthogonal complement of $W_0\subset W$, with respect to $\Psi_W$.  We set 
\[
\Psi_{W_0}:=\Psi_W|W_0\times W_0,\: \Psi_{W_0^\perp}:=\Psi_W|W^\perp_0\times W_0^\perp,\:\text{and} \: \Psi_0:=\Psi|(\R\times W_0)\times(\R\times W_0).
\]
By these conventions $\Psi$ partitions according to the decomposition $\R\times W_0\times W_0^\perp$ as
\[
\Psi=\left(\begin{array}{ccc}
\psi_0&0&0\\
0&-\Psi_{W_0}&0\\
0&0&-\Psi_{W_0^\perp}
\end{array}\right).
\]
Note that  $\Psi_0$ and  $\Psi_W$  partition with respect to the decompositions  $\R\times W_0$ and \  $W=W_0\times W_0^\perp$ accordingly as
\[
\Psi_0=\left(\begin{array}{cc}
\psi_0&0\\
0&-\Psi_{W_0}
\end{array}\right),\ \  \text{and}\ \ \Psi_W=\left(\begin{array}{cc}
\Psi_{W_0}&0\\
0&\Psi_{W_0^\perp}
\end{array}\right).
\]

\noindent
Since    $(\mu,\w_0)\in \mathcal{L}_{0}$ implies that   $\mu^2>\psi_{W_0}(\w)=\psi_{W}(\w)$, clearly,   $\mathcal{L}_{0}\subseteq\R\times W_0\subset\R\times W$ is a subcone of  $\mathcal{L}$, i.e.,  $\mathcal{L}_{0}\subseteq \mathcal{L}$.\\

\noindent
 Next we determine the group $G_{0}:=\{g\in G|g(\mathcal{L}_{0})\subseteq \mathcal{L}_{0}\}$.  Note that  
 \[
 \R_{++}\subset\R_{++}\times {\mathrm SO}^{\shortuparrow}(\Psi)
 \]
  is obviously contained in  $G_{0}$. 
 \begin{Proposition}\label{decomA}
Let \  $A\in{\mathrm SO}^{\shortuparrow}(\Psi)$.  Then \  $A(\mathcal{L}_{0})\subseteq \mathcal{L}_{0}$ if and only if $A$ partition as
\[
\left(\begin{array}{cc}
A_0&0\\
0&A_{W_0^\perp}
\end{array}\right)
\]
with

(i) \  $A_0\in{\mathrm O}(\Psi_0)$, i.e., \  $\Psi_0\circ(A_0\times A_0)=\Psi_0$,

(ii) \  $A_{W_0^\perp}\in{\mathrm O}(\Psi_{W_0^\perp})$, i.e., \  $\Psi_{W_0^\perp}\circ(A_{W_0^\perp}\times A_{W_0^\perp})=\Psi_{W_0^\perp}$,

(iii) \  $a_0>0$, and

(iv) \  ${\mathrm det}(A_{W_0})\cdot{\mathrm det}(A_{W_0^\perp})>0$.
\end{Proposition}
\begin{proof}
  Suppose  $A\in{\mathrm SO}^{\shortuparrow}(\Psi)$ acts invariantly on $\mathcal{L}_{0}$, i.e., $A(\mathcal{L}_{0})\subseteq \mathcal{L}_{0}$. Let
\[
A=\left(\begin{array}{c|c}
\begin{array}{cc}
a_0&\mathbf{a}_{0W_0}\\
\mathbf{a}_{W_00}&A_{W_0}\\
\end{array}
&\begin{array}{c}
\mathbf{a}_{0W_0^\perp}\\
A_{W_0W_0^\perp}
\end{array}\\
\hline
\begin{array}{cc}
\mathbf{a}_{W_0^\perp 0}&A_{W_0^\perp W_0}
\end{array}&A_{W_0^\perp}
\end{array}\right)=\left(\begin{array}{cc}
A_0&\mathbf{a}_{0W_0^\perp}\\
A_{W_0^\perp0W_{0}}&A_{W_0^\perp}
\end{array}\right)
\]

and

\[
\Psi=\left(\begin{array}{cc}
\Psi_0&0\\
0&-\Psi_{W_0^\perp}
\end{array}\right)
\]
\noindent
be the partitions of   $A$,  $A$, and  $\Psi$  with respect to the decompositions   $\R\times W_0\times W_0^\perp$, \  $(\R\times W_0)\times W_0^\perp$,
and   $(\R\times W_0)\times W_0^\perp,$  respectively.  Note that 
\[
A_{W_{0}^{\perp}0W_{0}}=(\mathbf{a}_{W_{0}^{\perp}0}, A_{W_{0}^{\perp}W_{0}}): \R\times W_{0}\rightarrow W_{0}^{\perp},
\]
 and $A(\mathcal{L}_{0})\subseteq \mathcal{L}_{0}$ implies that   $A_{W_{0}^{\perp}0W_{0}}(\lambda,\w_0)=0$  for all   $(\lambda,\w_0)\in\R\times W_0$. Therefore  $A_{W_{0}^{\perp}0W_{0}}=0$. Since  $A\in{\mathrm SO}^{\shortuparrow}(\Psi)$  we also have   $\Psi\circ (A\times A)=\Psi,$, i.e.,
\[
\left(\begin{array}{cc}
\Psi_0&0\\
0&-\Psi_{W_0^\perp}
\end{array}\right)\circ\left[\left(\begin{array}{cc} 
A_0&A_{0W_0^\perp}\\
0&A_{W_0^\perp}
\end{array}\right)\times\left(\begin{array}{cc}
A_0&A_{0W_0^\perp}\\
0&A_{W_0^\perp}
\end{array}\right)\right]=\left(\begin{array}{cc}
\Psi_0&0\\
0&-\Psi_{W_0^\perp}
\end{array}\right).
\]

This is equivalent to the followings.

\[
\begin{array}{ccccc}
({\mathrm i}&\Psi_0\circ(A_0\times A_0)&=&\Psi_0.\\
({\mathrm ii})& \Psi_0\circ(A_0\times A_{W_0^\perp})&=&0.\\
({\mathrm iii})& \Psi_{W_0^\perp}\circ(A_{W_0^\perp}\times A_{W_0^\perp})&=&\Psi_{W_0^\perp}.
\end{array}
\]

\end{proof}
 
\subsection{A maximal invariant statistic for testing the hypothesis \label{sub:test1}}
\noindent
Now consider the action of  $G_{0}$  on   $\mathcal{L}$ , i.e., the restriction of the (transitive) action of $G$  on $\mathcal{L}$ to  $G_{0}$.  We define the statistic
\begin{eqnarray}\label{e:maxinv}
   \notag  \mathfrak{m}:\mathcal{L}:&\rightarrow&\R_{++}\\
    (\lambda, \w_0,\w_{0}^{\perp})&\mapsto&\dfrac{\Psi_{W_0^\perp}(\w_{0}^{\perp},\w_{0}^{\perp})}{\Psi_{0}((\lambda,\w_0), (\lambda,\w_0))},
    \end{eqnarray}
where  $(\lambda,\w_0,\w_{0}^{\perp})$ represents a typical element of $\R\times W_0\bigoplus W_{0}^\perp$.  
\begin{Proposition}\label{prop:max_inv1}
The mapping $\mathfrak{m}$ in Eq. \eqref{e:maxinv} is a faithful representation of the orbit projection   $\pi:\mathcal{L}\rightarrow \mathcal{L}/G_{0}$ and therefore a maximal invariant statistic. 
\end{Proposition}
\begin{proof}
First for notational convenience set  $\|\w\|^2=\Psi_W(\w,\w),$   for  $\w\in W$. Thus
\[
\mathfrak{m}((\lambda, \w_0,\w_{0}^{\perp}))= \dfrac{\| \w_{0}^{\perp}\|^2}{\lambda^2-\| \w_0\|^2}, \ \forall (\lambda, \w_0,\w_{0}^{\perp})\in \mathcal{L}.
\]
 For each $(a,A)\in G_{0}$  by partitioning  $A$  as in Proposition \ref{decomA} we obtain
 \begin{eqnarray*}
       \mathfrak{m}((a,A)(\lambda, \w_0, \w_{0}^{\perp}))&=&\frac{\| aA_{W_0^\perp}\w_{0}^{\perp} \|^2}{(a\lambda)^2-\| aA_{W_0}\w_0 \|^2 }\\
                                   &=& \dfrac{\| A_{W_0^\perp}\w_{0}^{\perp} \|^2}{\lambda^2-\| A_{W_0}\w_0\ \|^2 }\\
                                   &=&\dfrac{\| \w_{0}^{\perp}\|^2}{\lambda^2-\| \w_0\|^2}.
 \end{eqnarray*}
This shows that $\mathfrak{m}$ is invariant under the action of  $G_{0}$  on  $\mathcal{L}$. 
Now suppose that
\[
\mathfrak{m}((\lambda,\w_0,\w_{0}^{\perp}))=\mathfrak{m}((\mu,\mathbf{u}_0,\mathbf{u}_0^{\perp})),
\]
 i.e.,
\[
  \frac{\| \w_{0}^{\perp}\|^2}{\lambda^2-\| \w_0\|^2}=\dfrac{\| \mathbf{u}_0^{\perp}\|^2}{\mu^2-\| \mathbf{u}_0\|^2}.
  \] 
  If $\mathbf{u}_0^{\perp}=0$ , then $\w_{0}^{\perp}=0$  and, since $G_{0}$ acts transitively on $\mathcal{L}_{0}$,  we can find   $(a,A)\in G_{0}$  such that   
  \[
  (a,A)(\lambda, \w_0,0)=(\mu,\mathbf{u}_0,0),
  \]
  which shows that $(\mu,\mathbf{u}_0,\mathbf{u}_0^{\perp})$ and  $(\mu,\w_0,\w_0^{\perp})$ are in the same $G_{0}$-orbit.\\
  \noindent
  Now assume  $\w_{0}^{\perp}\neq0$, and therefore$\mathbf{u}_0^{\perp}\neq 0$.  We set
  \[
  a:=\frac{\| \mathbf{u}_0^{\perp}\|}{\| \w_{0}^{\perp}\|}=\sqrt{\frac{\mu^2-\| \mathbf{u}_0\|^2}{\lambda^2-\| \w_0\|^2}}.
  \]
  Choose  $A_{0}\in SO^{+}(\Psi_0)$  and   $A_{W_0^\perp}\in SO(\Psi_{W_0^\perp})$  such that
  \[
  A_{0}(\frac{(\lambda,\w_0)}{\sqrt{\lambda^2-\| \w_0\|^2}})=\frac{(\mu,\mathbf{u}_0)}{\sqrt{\mu^2-\| \mathbf{u}_0\|^2}},
  \]
  and
   \[
   A_{W_0^\perp}(\frac{\w_{0}^{\perp}}{\| \w_{0}^{\perp}\|})=\frac{\mathbf{u}_0^{\perp}}{\| \mathbf{u}_0^{\perp}\|}.
   \]
  For $A:=\left(\begin{array}{cc}
A_0&0\\
0&A_{W_0^\perp}
\end{array}\right)$ we have $(a, A)\in G_{0}$  and
\begin{eqnarray*}
  (a, A)(\lambda,\w_0,\w_{0}^{\perp})&=&a(A_0(\lambda,\w_0),A_{W_0^\perp}(\w_{0}^{\perp}))  \\
   &=&a(\sqrt{\dfrac{\lambda^2-\| \w_0\|^2}{\mu^2-\| \mathbf{u}_0\|^2}}(\mu,\mathbf{u}_0),\dfrac{\| \w_{0}^{\perp}\|}{\| \mathbf{u}_0^{\perp} \|}\mathbf{u}_0^{\perp}) \\
  &=&a(a^{-1}(\mu,\mathbf{u}_0),a^{-1}\mathbf{u}_0^{\perp})\\
  &=&(\mu,\mathbf{u}_0,\mathbf{u}_0^{\perp}).
\end{eqnarray*}
\end{proof}
 \begin{Remark}\label{rem:projection}
 From statistical point of view,  it is more useful to write the maximal invariant statistics $\mathfrak{m}$ in Eq.  \eqref{e:maxinv} as
 \[\mathfrak{m}(\lambda, \mathbf{w})=\dfrac{\|\w-\mathfrak{p}(\w)\|^2}{\lambda^2-\|\mathfrak{p}(\w)\|^2}\quad \forall (\lambda, \mathbf{w})\in \mathcal{L},
 \]
 where  $\mathfrak{p}:W\rightarrow W_{0}$  is the orthogonal projection of  $W$ onto  $W_{0}$.
 \end{Remark}
\subsection{Testing the hypothesis \eqref{eq:test1}}
\noindent
Finally we are in the position to give a test statistic for testing the hypothesis \eqref{eq:test1}. First note that the hypothesis \eqref{eq:test1}  is invariant under $G_{0}$.
  \begin{Theorem}\label{prop:IND}
 Consider  $\mathcal{W}^{\mathcal{L}}_{\eta,\iota(\sigma)}$, the Wishart distribution on $\mathcal{L}$, where $\sigma\in \mathcal{L}_{0}$ and $\iota(\sigma)\in \mathcal{L}$ is the embedding
   of  $\sigma$ into   $\mathcal{L}$.  Let $t:\mathcal{L}\rightarrow \mathcal{L}_{0}$ be the mapping  $(y,\mathbf{z})\mapsto (y,\mathfrak{p}(\mathbf{z}))$, where  $(y,\mathbf{z})\in \mathcal{L}$. Then for the mapping 
 \begin{eqnarray*}
   (t,\mathfrak{m}):\mathcal{L} &\rightarrow& \mathcal{L}_{0}\times \R_{++} \\
   (y,\mathbf{z}) &\mapsto&((y,\mathfrak{p}(\mathbf{z})),\mathfrak{m}(y,\mathbf{z}))
   \end{eqnarray*}
   we have \  $(t,m)(\mathcal{W}^{\mathcal{L}}_{\eta,\iota(\sigma)})=\mathcal{W}^{\mathcal{L}_{0}}_{\eta,\sigma}\bigotimes
   \mathfrak{m}(\mathcal{W}^{\mathcal{L}}_{\eta,\iota(\sigma)})$.
      \end{Theorem}
   \begin{proof} First write $\mathcal{W}^{\mathcal{L}}_{\eta,\iota(\sigma)}$ as a density with respect to
   the  invariant measure
   \[
   d\nu(y,\mathbf{z})=(y^2-\|\mathbf{z}\|^2)^{\eta-\frac{m+1}{2}}dyd\mathbf{z},
   \]
    and rewrite the density in terms of  $t(y,\mathbf{z})$ and  $\mathfrak{m}(y,\mathbf{z})$. Then
   apply \cite[Lemma 3]{Andersson1983} to $\nu$ and state the transformation result.
   \end{proof}
 \begin{Corollary}\label{maxdist}  Let  $m_0:=\mathrm{dim_{\R}}(W_0)$ and  $m_1=m-m_0$. Then the transformed measure  $\mathfrak{m}(\mathcal{W}^{\mathcal{L}}_{\eta,\iota(\mathbf{\sigma}_1)})$   is  the beta distribution   $\beta(\eta-\frac{m-1}{2}, \frac{m_1}{2})$.
 \begin{proof}
 We start with writing \  $(t,m):\mathcal{L} \rightarrow \mathcal{L}_{0}\times
 \R_{++}$ as the composition of the mappings
 \[
 \mathcal{L}\rightarrow \R_{++}\times W_0\oplus W_0^\perp\rightarrow\R_{++}\times
 W_0\times \R_{++}
 \]
 \begin{eqnarray*}
 (y,\mathbf{z})&\mapsto& (y,\mathbf{z}_0, \mathbf{z}_1):=\left(y,\mathfrak{p}(\mathbf{z}),\frac{\mathbf{z}-\mathfrak{p}(\mathbf{z})}{\sqrt{y^2-\|\mathfrak{p}(\mathbf{z})\|^2}}\right)\\
  &\mapsto& (y,\mathbf{z}_0,u):= (y,\mathbf{z}_0,\|\mathbf{z}_1\|^2).
   \end{eqnarray*}
  Using these compositions we transfer the probability density of the Wishart distribution $\mathcal{W}^{\mathcal{L}}_{\eta,\iota(\sigma)}$ as follows:
\begin{eqnarray*}
   d\mathcal{W}^{\mathcal{L}}_{\eta,\iota(\sigma)}(y,\mathbf{z})&\rightarrow&
   \dfrac{k(m_0,\eta)}{k(m,\eta)}(1-\|\mathbf{z}_1\|^2)^
   {\eta-\frac{m+1}{2}}1_{\{\|\mathbf{z}_1\|<1\}}
   (\mathbf{z}_1)d\mathcal{W}^{\mathcal{L}_0}_{\eta,\sigma}(y,\mathbf{z}_0)d\mathbf{z}_1\\
   &\rightarrow&\dfrac{k(m_0,\eta)\pi^{\frac
   {m_0}{2}}}{k(m,\eta)
   \Gamma(\frac{m_1}{2})}(1-u)^
   {\eta-\frac{m+1}{2}}u^{\frac{m_1}{2}-1}
   1_{[0,1]}(u)d\mathcal{W}^{\mathcal{L}_0}_{\eta,\sigma}(y,\mathbf{z}_0)du\\
   &=&\dfrac{\Gamma(\eta-\frac{m_0-1}{2})}{\Gamma(\eta-\frac{m-1}{2})\Gamma(\frac{m_1}{2})}(1-u)^
   {\eta-\frac{m+1}{2}}u^{\frac{m_1}{2}-1}
   1_{[0,1]}(u)d\mathcal{W}^{\mathcal{L}_0}_{\eta,\sigma}(y,\mathbf{z}_0)du\\
   &=&d\mathcal{W}^{\mathcal{L}_0}_{\eta,\sigma}(y,\mathbf{z}_0)d\beta(\eta-\frac{m-1}{2}, \frac{m_1}{2})(u).
\end{eqnarray*}
Thus  $\mathfrak{m}(y,\mathbf{z})\sim \beta(\eta-\frac{m-1}{2}, \frac{m_1}{2})$.
 \end{proof}
 \end{Corollary}
 \begin{Proposition}\label{prop:test1}
 The likelihood ratio \textbf{ LR} statistic   $Q$  for hypothesis \eqref{eq:test1} is given by
 \begin{equation}\label{eq:LR}
    Q=(\dfrac{y^2-\|\mathbf{z}\|^2}{y^2-\|\mathfrak{p}(\mathbf{z})\|^2})^\eta=\mathfrak{m}(y,\mathbf{z})^{\eta}.
\end{equation}
Moreover,  $Q$  is independent $\mathcal{W}^{\mathcal{L}_0}_{\eta,\sigma}$ and  $ Q^{\frac{1}{\eta}}\sim \beta(\eta-\frac{m-1}{2},\frac{m-m_0}{2})$.
\end{Proposition}
\begin{proof}
  First note that for the submodel \eqref{eq:submodel} and  the observation    $(y,\mathbf{z}),$  the \textbf{ML} estimator of \  $\sigma \in \mathcal{L}_{0}$ becomes
\[
\widehat{\sigma}_1((y,\mathbf{z}))=(y,\mathfrak{p}(\mathbf{z})),
\]
where  $\mathfrak{p}$ is the orthogonal projection on the subspace  $W_0$. Thus {\bf LR} is given by Eq. \eqref{eq:LR}. The rest of the proof are direct consequences of Proposition \ref{prop:IND} and  Corollary \ref{maxdist}.
\end{proof}
\noindent
Note that the Proposition \ref{prop:test1} reduces to Theorem 5 in \cite{Jensen1988} for choice of  $\eta=N/4$.
 \section{The Bartlett's test for Lorentz Wishart Models}\label{sec:bart}
 \noindent
 As before, suppose $\eta $  is known, and  consider the statistical model
 \[
 \left\{\mathcal{W}^{\mathcal{L}}_{\eta,\sigma_1}\otimes \mathcal{W}^{\mathcal{L}}_{\eta,\sigma_2}:\: (\sigma_1,\sigma_2)\in \mathcal{L}\times \mathcal{L}\right\}
  \]
  and its submodel 
  \[
    \left\{\mathcal{W}^{\mathcal{L}}_{\eta,\sigma}\otimes \mathcal{W}^{\mathcal{L}}_{\eta, \sigma}:\: \sigma\in \mathcal{L}\right\}.
  \]
  Consider the hypothesis 
  \begin{equation}\label{eq:test2}
 \tag{T2}    H_0 : \sigma_1=\sigma_2=\sigma \quad   \text{vs.}\quad   H: \sigma_1\neq\sigma_2.
  \end{equation}
  
 \subsection{A maximal invariant statistic associated with the hypothesis \eqref{eq:test2}}
 \noindent
 Note that $\mathcal{L}$ can be considered a subcone of $\mathcal{L}\times \mathcal{L}$, via diagonal
 embedding. Also the action of   $G$  on $\mathcal{L}$ can be, canonically, extended to an action on  $\mathcal{L}\times \mathcal{L}$ given by
\begin{equation}\label{e:action}
g(\sigma_1,\sigma_2)=(g\sigma_1,g\sigma_2) \qquad \forall g\in G, \;\forall(\sigma_1,\sigma_2)\in \mathcal{L}\times\mathcal{L}.
 \end{equation}
 Under this consideration the hypothesis \eqref{eq:test2} is invariant under $G$, and therefore a maximal invariant statistic is desired. To obtain a maximal invariant we proceed with the following lemma.
   \begin{Lemma}\label{lemma} For every  $\sigma_1$ and  $\sigma_2 \in \mathcal{L}$
    \begin{eqnarray}
    \Psi(\sigma_1,\sigma_2)\geq
    \sqrt{\Psi(\sigma_1,\sigma_1)\Psi(\sigma_2,\sigma_2)}\,,
    \end{eqnarray}
    and the equality holds if and only if   $g\sigma_1=\sigma_2$  for some  $g\in G$.
    \end{Lemma}
    \begin{proof}
    Let \  $\sigma_1\cdot\sigma_2:=\Psi(\sigma_1,\sigma_2)$, where
    \  $\sigma_1=(\lambda, \mathbf{w})$ and \  $\sigma_2=(\mu,\mathbf{u})$. Then
    \begin{eqnarray*}
      (\sigma_1\cdot \sigma_2)^2-\|\sigma_1\|^2\|\sigma_2\|^2&=&
      (\lambda\mu-\w.\mathbf{u})^2-(\lambda^2-\|\w\|^2)(\mu^2-\|\mathbf{u}\|^2)\\
      &\geq&\lambda^2\|\mathbf{u}\|^2-2\lambda\mu\|\mathbf{u}\|\|\w\|+\mu^2\|\w\|^2\\
      &=&(\lambda\|\mathbf{u}\|-\mu\|\w\|)^2\geq 0.
    \end{eqnarray*}
    The equality holds if and only if \  $\lambda\|\mathbf{u}\|=\mu\|\w\|$ which
    is equivalent to \  $\|\sigma_2\|=\dfrac{\mu}{\lambda}\|\sigma_1\|$.

    \end{proof}
    \begin{Notation}
   In the remainder of this paper for brevity we use the notation   $\sigma_1\cdot\sigma_2$  for the quadratic product  $\Psi(\sigma_1,\sigma_2)$.
   \end{Notation}
 \begin{Proposition} \label{prop:max_inv2}
 The mapping $\pi:\mathcal{L}\times \mathcal{L}\rightarrow\R_{++}^2$\ with \  $(\sigma_1, \sigma_2)\mapsto (\xi_1, \xi_2)$, where

    \begin{eqnarray*}
    \xi_1&:=&\dfrac{\sigma_1\cdot\sigma_2+
    \sqrt{(\sigma_1\cdot\sigma_2)^2-\|\sigma_1\|^2
    \|\sigma_2\|^2}}{\|\sigma_1\|^2} \quad \text{and}\\
    \xi_2&:=&\dfrac{\sigma_1\cdot \sigma_2-\sqrt{(\sigma_1\cdot \sigma_2)^2-\|\sigma_1\|^2
    \|\sigma_2\|^2}}{\|\sigma_1\|^2},
    \end{eqnarray*}
    is  maximal invariant under the action of $G$, defined in  Eq. \eqref{e:action}, and $\xi_1$,  $\xi_2$ are eigenvalues of  $\sigma_2$  with respect to  $\sigma_1$.
    \end{Proposition}
 \begin{proof}
 By Lemma \ref{lemma} the mapping   $\pi$ is well-defined. Next we show  that  $\xi_1\geq \xi_2$ are eigenvalues
 of    $\sigma_1$  with respect to   $\sigma_2$. Note that the characteristic polynomial of $\sigma_2$ with respect to $\sigma_1$ is 
 \begin{equation}\label{chareq}
  p(\ell):=\det(\sigma_2-\ell\sigma_1)=\|\sigma_1\|^2\ell^2-2(\sigma_1\cdot\sigma_2) \ell-\|\sigma_2\|^2.
 \end{equation}
 One then can easily check that  $\xi_1$  and   $\xi_2$  are indeed the roots of $p(\ell)$.  Next we show that $\pi$  is, moreover, onto. Let the real numbers  $r_1\geq r_2\geq 0$ be given.  Set  $\lambda :=\dfrac{r_1+r_2}{2}$. Choose a vector $ \w\in W$ such that $ \|\w\|=\dfrac{r_1-r_2}{2}$. Let  $\sigma=(\lambda, \w)$ and $e= (1,0)$. Then $\pi((\sigma, e))=(r_1,r_2)$ as desired. It is clear that  $\pi$  is invariant under  the action of  $G$. Suppose $\pi(\sigma_1,\sigma_2)=\pi(\sigma'_1,\sigma'_2)$. Since  $G$  acts transitively on  $\mathcal{L}$ and $\pi$  is invariant under the action of $G$, without loss  of generality,  we may assume  that $\sigma_2=\sigma'_2=e$.  We have
  \[
  \pi(\sigma_1,e)=\pi(\sigma'_1,e).
  \]
  Eq.\eqref{chareq} implies that the characteristic polynomials $p_1(\ell):=\|\sigma_1\|^2\ell^2-2(\sigma_1.e) \ell-1$  and  $p_2(\ell):=\|\sigma'_1\|^2\ell^2-2(\sigma'_1.e) \ell-1$  are identical.Therefore,  if    $\sigma_1=(\lambda, \mathbf{w})$  and  $\sigma'_1=(\lambda',\mathbf{u})$, then
  \begin{eqnarray*}
  \lambda&=&\sigma_1\cdot e=\sigma'_1\cdot e=\lambda'\\
  \lambda^2-\|\w\|^2&=&\|\sigma_1\|^2=\|\sigma'\|^2=\lambda'^2-\|\mathbf{u}\|^2.
  \end{eqnarray*}
   Thus  $(\sigma_1,e)$  and  $(\sigma'_1,e)$  are in the same $G$-orbit.
  \end{proof}
  \noindent
  
\subsection{Testing the hypothesis \eqref{eq:test2}}
\noindent
Now by using the maximal invariant obtained in Proposition \ref{prop:max_inv2} we can test the hypothesis \eqref{eq:test2} as follows. 
 \begin{Theorem}\label{btest}
  For the observation \ $(\tau_1,\tau_2) \in \mathcal{L}\times \mathcal{L}$,  the \textbf{ML} estimator of  $\sigma$ under $H_0$ is
\[
\widehat{\sigma}(\tau_1,\tau_2):=\dfrac{\tau_1+\tau_2}{2},
\]
 and the \textbf{LR} statistic for testing hypothesis \eqref{eq:test2} is
 \[
\left(16\prod_{j=1}^2\dfrac{\xi_j}{(1+\xi_j)^2}\right)^{\eta},
\]
  where  $\xi_1>\xi_2$  are eigenvalues of  $\tau_2$ with respect to   $\tau_1$. Furthermore, under the hypothesis $H_0$  the statistics  $\widehat{\sigma}(\tau_1,\tau_2)$  and   $\pi(\tau_1,\tau_2)=(\xi_1,\xi_2)$   are independently distributed,  $\widehat{\sigma}(\tau_1,\tau_2)\sim \mathcal{W}^{\mathcal{L}}_{\eta,\sigma }$\  and the density of $\pi(\tau_1,\tau_2)=(\xi_1,\xi_2)$  is given by

 \begin{equation*}\label{Lotrentzb}
\dfrac{(2\pi)^{n-2}}{B_{\mathcal{L}}(\eta,\eta)\Gamma_{\mathcal{L}}(n-2)}
(\xi_1-\xi_2)^2\frac{(\prod_{j=1}^2\xi_j)^{\eta-\frac{n}{2}}}{(\prod_{j=1}^2(1+\xi_j))^{2\eta}}.
\end{equation*}
  \end{Theorem}
  \begin{proof}
   This follows from Proposition \ref{prop:max_inv2} and Theorem 6.1 in \cite{Bendavid2010} when the irreducible cone $\Omega$ is the Lorentz cone $\mathcal{L}$. 
 \end{proof}
 We should mention that the proof of Theorem 6.1 \cite{Bendavid2010}, which simultaneously applies to all five types of irreducible cones, heavily rests on the analysis of simple Euclidean Jordan algebras.
 \begin{Remark}
 Recall that in the classical multivariate statistics, the Bartlett's test is testing
\begin{equation}\label{bartest}
 H_0: \Sigma=\sigma^2 I_n\quad\text{vs.}\quad H: \Sigma\neq \sigma^2 I_n,
\end{equation} 
for a Gaussian model, where the sample space is $\R^n,$ the distributions are multivariate normal distribution $\N_n(0,\Sigma)$ and the parameter space is  ${\mathrm PD}_{n}(\R)$. Therefore the Bartlett's test is testing whether $n$ univariate Gaussian distributions are independent and have the same variance $\sigma$.
 \end{Remark}
\section{Closing Remarks}
In this paper we have shown how maximal invariant statistics can be derived and used for testing two specific invariant statistical hypotheses for Lorentz Wishart models. Analogs of such hypotheses have been already studied for real, complex and quaternion type Wishart models \cite{Andersson1983}. An interesting topic of future research, which its analog has been studied in \cite{Andersson1983},  is testing the hypothesis that the scale parameter $\Sigma$  of the real Wishart distribution has a Lorentz structure. To clarify what is meant by a Lorentz type we note that every Lorenz cone $\mathcal{L}$  is isomorphic to a subcone of  ${\mathrm PD}_{n}(\R)$ for a suitable $n$ (see \cite{Faraut1994}, \cite{Jensen1988} for detail), which means that there exist a linear injection $\rho: \R\times W\rightarrow \R^{n\times n}$  such that $\rho(\lambda, \w)$ is positive definite, for each $(\lambda, \w)\in \mathcal{L}$. Therefore testing whether $\Sigma$  has a Lorentz structure requires to show that $\Sigma=\rho(\lambda, \w)$  for some $(\lambda, \w)\in \mathcal{L}$.

\bibliographystyle{plain}

\end{document}